\documentclass[reqno]{amsart}
\usepackage{setspace,amssymb}
\usepackage{ifpdf}
\ifpdf
 \usepackage[hyperindex]{hyperref}
\else
 \expandafter\ifx\csname dvipdfm\endcsname\relax
 \usepackage[hypertex,hyperindex]{hyperref}
 \else
 \usepackage[dvipdfm,hyperindex]{hyperref}
 \fi
\fi
\allowdisplaybreaks[4]
\numberwithin{equation}{section}
\theoremstyle{plain}
\newtheorem{thm}{Theorem}[section]
\newtheorem{cor}{Corollary}[section]
\theoremstyle{remark}
\newtheorem{rem}{Remark}[section]
\DeclareMathOperator{\td}{d\mspace{-2mu}}

\begin{document}

\title[Explicit formulas for Bernoulli and Stirling numbers]
{Explicit formulas for computing Bernoulli numbers of the second kind and\\ Stirling numbers of the first kind}

\author[F. Qi]{Feng Qi}
\address{College of Mathematics, Inner Mongolia University for Nationalities, Tongliao City, Inner Mongolia Autonomous Region, 028043, China}
\email{\href{mailto: F. Qi <qifeng618@gmail.com>}{qifeng618@gmail.com}, \href{mailto: F. Qi <qifeng618@hotmail.com>}{qifeng618@hotmail.com}, \href{mailto: F. Qi <qifeng618@qq.com>}{qifeng618@qq.com}}
\urladdr{\url{http://qifeng618.wordpress.com}}



\begin{abstract}
In the paper, by establishing a new and explicit formula for computing the $n$-th derivative of the reciprocal of the logarithmic function, the author presents new and explicit formulas for calculating Bernoulli numbers of the second kind and Stirling numbers of the first kind. As consequences of these formulas, a recursion for Stirling numbers of the first kind and a new representation of the reciprocal of the factorial $n!$ are derived. Finally, the author finds several identities and integral representations relating to Stirling numbers of the first kind.
\end{abstract}

\thanks{This paper was typeset using \AmS-\LaTeX}

\maketitle

\section{Introduction}

It is general knowledge that the $n$-th derivative of the logarithmic function $\ln x$ for $x>0$ is
\begin{equation}
(\ln x)^{(n)}=(-1)^{n-1}\frac{(n-1)!}{x^n}
\end{equation}
for $n\in\mathbb{N}$, where $\mathbb{N}$ denotes the set of all positive integers.
One may ask a question: What is the formula for the $n$-th derivative of the reciprocal of the logarithmic function $\ln x$? There have been some literature to deal with this question. For example,
Lemma~2 in~\cite{Liu-Qi-Ding-2010-JIS} reads that for any $m\ge0$ we have
\begin{equation}\label{Liu-Qi-Ding-2010-JIS-log-der}
\biggl[\frac1{\ln(1+t)}\biggr]^{(m)} =\frac1{(1+t)^m}\sum_{i=0}^m(-1)^ii!\frac{s(m,i)}{[\ln(1+t)]^{i+1}},
\end{equation}
where $s(n,k)$ are Stirling numbers of the first kind, which are defined by
\begin{equation}\label{s(n-k)-1dfn}
\frac{[\ln(1+x)]^m}{m!}=\sum_{k=m}^\infty\frac{s(k,m)}{k!}x^k,\quad |x|<1.
\end{equation}
\par
The first aim of this paper is to establish a new and explicit formula for computing the $n$-th derivative of the reciprocal of the logarithmic function. As consequences of this formula, a recursion for Stirling numbers of the first kind and a new representation of the reciprocal of the factorial $n!$ are derived.
\par
The Bernoulli numbers $b_0,b_1,b_2,\dotsc,b_n,\dotsc$ of the second kind may be defined by
\begin{equation}\label{bernoulli-second-dfn}
\frac{x}{\ln(1+x)}=\sum_{n=0}^\infty b_nx^n.
\end{equation}
The first few Bernoulli numbers $b_n$ of the second kind are
\begin{align}
b_0&=1,& b_1&=\frac12,& b_2&=-\frac1{12},& b_3&=\frac1{24},& b_4&=-\frac{19}{720},& b_5&=\frac3{160}.
\end{align}
For more information, please refer to~\cite{Howard-Fibon-1994, Jordan-1965-book} and closely related references therein. By the way, we note that the so-called Cauchy number of the first kind may be defined by $n!b_n$. See~\cite{Agoh-Dilcher-2010, Liu-Qi-Ding-2010-JIS} and plenty of references cited therein.
One may also ask a natural question:
Can one discover an explicit formula for computing $b_n$ for $n\in\mathbb{N}$?
There have been several formulas and recurrence relations for computing $b_n$. For example, it is derived in~\cite{Nemes-JIS-2011} that
\begin{equation}
b_n=\frac1{n!}\sum_{k=0}^n\frac{s(n,k)}{k+1},
\end{equation}
where $s(n,k)$ is defined by
\begin{equation}
\prod_{k=0}^{n-1}(x-k)=\sum_{k=0}^ns(n,k)x^k. \label{gen-funct-2}
\end{equation}
\par
We remark that two definitions of $s(n,k)$ by~\eqref{s(n-k)-1dfn} and~\eqref{gen-funct-2} are coincident.
\par
The second aim of this paper is to derive a new and explicit formula for calculating Bernoulli numbers $b_n$ of the second kind.
\par
Finally, we will find several identities and integral representations relating to Stirling numbers of the first kind $s(n,k)$.

\section{Explicit formula for derivatives of the logarithmic function}

In this section, we establish a new and explicit formula for  computing the $n$-th derivative of the reciprocal of the logarithmic function, which will be applied in next section to derive an explicit formula for calculating Bernoulli numbers of the second kind.

\begin{thm}\label{reciprocal=log=der=thm}
For $n\in\mathbb{N}$, we have
\begin{equation}\label{reciprocal=log=der=eq}
\biggl(\frac1{\ln x}\biggr)^{(n)}=\frac{(-1)^n}{x^n}\sum_{i=2}^{n+1}\frac{a_{n,i}}{(\ln x)^{i}},
\end{equation}
where
\begin{equation}\label{a=n=0}
a_{n,2}=(n-1)!
\end{equation}
and, for $n+1\ge i\ge3$,
\begin{equation}\label{a=n=i=eq}
a_{n,i}=(i-1)!(n-1)!\sum_{\ell_1=1}^{n-1} \frac1{\ell_1}\sum_{\ell_2=1}^{\ell_1-1}\frac1{\ell_2}\dotsm \sum_{\ell_{i-3}=1}^{\ell_{i-4}-1}\frac1{\ell_{i-3}} \sum_{\ell_{i-2}=1}^{\ell_{i-3}-1}\frac1{\ell_{i-2}}.
\end{equation}
\end{thm}

\begin{proof}
An easy differentiation gives
\begin{multline*}
\biggl(\frac1{\ln x}\biggr)^{(n+1)}=\Biggl[\biggl(\frac1{\ln x}\biggr)^{(n)}\Biggr]'
=\Biggl[\frac{(-1)^n}{x^n}\sum_{i=2}^{n+1}\frac{a_{n,i}}{(\ln x)^i}\Biggr]'\\
=(-1)^n\sum_{i=2}^{n+1}a_{n,i}\biggl[\frac1{x^n(\ln x)^i}\biggr]'
=\frac{(-1)^{n+1}}{x^{n+1}}\sum_{i=2}^{n+1}a_{n,i}\frac{i+n\ln x}{(\ln x)^{i+1}}\\
=\frac{(-1)^{n+1}}{x^{n+1}} \Biggl[\sum_{i=2}^{n+1}\frac{ia_{n,i}}{(\ln x)^{i+1}} +\sum_{i=2}^{n+1}\frac{na_{n,i}}{(\ln x)^{i}}\Biggr]\\
=\frac{(-1)^{n+1}}{x^{n+1}} \Biggl[\sum_{i=3}^{n+2}\frac{(i-1)a_{n,i-1}}{(\ln x)^{i}} +\sum_{i=2}^{n+1}\frac{na_{n,i}}{(\ln x)^{i}}\Biggr]\\
=\frac{(-1)^{n+1}}{x^{n+1}} \Biggl[\frac{na_{n,2}}{(\ln x)^2}+ \sum_{i=3}^{n+1}\frac{(i-1)a_{n,i-1}+na_{n,i}}{(\ln x)^i}+\frac{(n+1)a_{n,n+1}}{(\ln x)^{n+2}}\Biggr].
\end{multline*}
Equating coefficients of $(\ln x)^i$ for $2\le i\le n+2$ on both sides of
\begin{multline*}
\frac{(-1)^{n+1}}{x^{n+1}}\sum_{i=2}^{n+2}\frac{a_{n+1,i}}{(\ln x)^{i}}\\
=\frac{(-1)^{n+1}}{x^{n+1}} \Biggl[\frac{na_{n,2}}{(\ln x)^2}+ \sum_{i=3}^{n+1}\frac{(i-1)a_{n,i-1}+na_{n,i}}{(\ln x)^i}+\frac{(n+1)a_{n,n+1}}{(\ln x)^{n+2}}\Biggr]
\end{multline*}
yields the recursion formulas of the coefficients $a_{n,i}$ satisfying
\begin{align}
a_{n+1,2}&=na_{n,2},\label{a=n+1=n=0-eq}\\
a_{n+1,n+2}&=(n+1)a_{n,n+1},\label{a=n+1=n=n-1-eq}
\end{align}
and
\begin{equation}\label{gener-recursion}
a_{n+1,i}=(i-1)a_{n,i-1}+na_{n,i}
\end{equation}
for $3\le i\le n+1$.
\par
From
\begin{equation*}
\biggl(\frac1{\ln x}\biggr)'=-\frac1{x(\ln x)^2},
\end{equation*}
it follows that
\begin{equation}\label{a=1-0-eq}
a_{1,2}=1.
\end{equation}
Combining~\eqref{a=1-0-eq} with~\eqref{a=n+1=n=0-eq} and~\eqref{a=n+1=n=n-1-eq} respectively results in~\eqref{a=n=0} and
\begin{equation}\label{a=n=n+1=n!}
a_{n,n+1}=n!.
\end{equation}
\par
Letting $i=3$ in~\eqref{gener-recursion} and using~\eqref{a=n=0} produce
\begin{equation}\label{recuring=n+1=n-1-eq}
a_{n+1,3}=2a_{n,2}+na_{n,3}=2(n-1)!+na_{n,3}
\end{equation}
for $n\ge2$. Utilizing~\eqref{a=n=n+1=n!} for $n=2$ as an initial value and recurring~\eqref{recuring=n+1=n-1-eq} figure out
\begin{equation}\label{a=n=3=recurs}
a_{n,3}=2!(n-1)!\sum_{k=1}^{n-1}\frac1k
\end{equation}
for $n\ge2$.
\par
Taking $i=4$ in~\eqref{gener-recursion} and employing~\eqref{a=n=3=recurs} give
\begin{equation}\label{recuring=n+1=n-2-eq}
a_{n+1,4}=3a_{n,3}+na_{n,4}=3\times2(n-1)!\sum_{k=1}^{n-1}\frac1k+na_{n,4}
\end{equation}
for $n\ge3$. Making use of~\eqref{a=n=n+1=n!} for $n=3$ as an initial value and recurring~\eqref{recuring=n+1=n-2-eq} reveal
\begin{equation}\label{a=n=4=recurs}
a_{n,4}=3!(n-1)!\sum_{i=1}^{n-1}\frac1i\sum_{k=1}^{i-1}\frac1k
\end{equation}
for $n\ge3$.
\par
By similar arguments to the deduction of~\eqref{a=n=3=recurs} and~\eqref{a=n=4=recurs}, we have
\begin{equation}\label{n=n-4-eq}
a_{n,5}=4!(n-1)!\sum_{j=1}^{n-1}\frac1j\sum_{i=1}^{j-1}\frac1i\sum_{k=1}^{i-1}\frac1k
\end{equation}
for $n\ge4$ and
\begin{equation}\label{n=n-5-eq}
a_{n,6}=5!(n-1)!\sum_{\ell=1}^{n-1}\frac1\ell\sum_{j=1}^{\ell-1} \frac1j\sum_{i=1}^{j-1}\frac1i\sum_{k=1}^{i-1}\frac1k
\end{equation}
for $n\ge5$.
\par
From~\eqref{a=n=3=recurs}, \eqref{a=n=4=recurs}, \eqref{n=n-4-eq}, and~\eqref{n=n-5-eq}, we inductively conclude the formula~\eqref{a=n=i=eq}. The proof of Theorem~\ref{reciprocal=log=der=thm} is thus completed.
\end{proof}

\begin{cor}
The coefficients $a_{n,i}$ in~\eqref{reciprocal=log=der=eq} satisfies the recursion~\eqref{gener-recursion} for $3\le i\le n+1$.
\end{cor}

\begin{proof}
This follows from the proof of Theorem~\ref{reciprocal=log=der=thm}.
\end{proof}

\begin{cor}
For $n\in\mathbb{N}$, the factorial $n!$ meets
\begin{equation}
\frac1{n!}=\sum_{\ell_1=1}^{n} \frac1{\ell_1}\sum_{\ell_2=1}^{\ell_1-1}\frac1{\ell_2}\dotsm \sum_{\ell_{n-1}=1}^{\ell_{n-2}-1}\frac1{\ell_{n-1}} \sum_{\ell_{n}=1}^{\ell_{n-1}-1}\frac1{\ell_{n}}.
\end{equation}
\end{cor}

\begin{proof}
This follows from combining~\eqref{a=n=i=eq} and~\eqref{a=n=n+1=n!} and simplifying.
\end{proof}

\begin{cor}
Stirling numbers of the first kind $s(n,i)$ for $1\le i\le n$ may be computed by
\begin{equation}
s(n,i)=(-1)^{n+i}(n-1)!\sum_{\ell_1=1}^{n-1} \frac1{\ell_1}\sum_{\ell_2=1}^{\ell_1-1}\frac1{\ell_2}\dotsm \sum_{\ell_{i-2}=1}^{\ell_{i-3}-1}\frac1{\ell_{i-2}} \sum_{\ell_{i-1}=1}^{\ell_{i-2}-1}\frac1{\ell_{i-1}},
\end{equation}
where $s(n,i)$ are defined by~\eqref{s(n-k)-1dfn} or~\eqref{gen-funct-2}.
\end{cor}

\begin{proof}
This is a direct consequence of comparing the formulas~\eqref{Liu-Qi-Ding-2010-JIS-log-der} and~\eqref{reciprocal=log=der=eq} and rearranging.
\end{proof}

\begin{cor}\label{1st-Stirling-recurs=cor}
For $1\le i\le n$, Stirling numbers of the first kind $s(n,i)$ satisfies the recursion
\begin{equation}\label{1st-Stirling-recurs}
s(n+1,i)=s(n,i-1) -ns(n,i).
\end{equation}
\end{cor}

\begin{proof}
Comparing formulas~\eqref{Liu-Qi-Ding-2010-JIS-log-der} and~\eqref{reciprocal=log=der=eq} reveals that
\begin{equation}\label{a(n-i)=s(n=n-1)-eq}
a_{n,i}=(-1)^{n+i-1}(i-1)!s(n,i-1)
\end{equation}
for $2\le i\le n+1$. Substituting this into~\eqref{gener-recursion} and simplifying lead to~\eqref{1st-Stirling-recurs}.
\end{proof}

\begin{rem}
The recursion~\eqref{1st-Stirling-recurs} is called in~\cite[p.~101]{Agoh-Dilcher-2010-integer} the ``triangular'' relation which is the most basic recurrence. Corollary~\ref{1st-Stirling-recurs=cor} recovers this triangular relation.
\end{rem}

\begin{rem}
It is helps to include a table of concrete values of the coefficients $a_{n,i}$ for small $n$. See Table~\ref{coefficients}.
\begin{table}[htbp]
\caption{The coefficients $a_{n,i}$}\label{coefficients}
\begin{tabular}{|c||c|c|c|c|c|}
  \hline
 $a_{n,i}$  & $i=2$ & $i=3$ & $i=4$ & $i=5$ & $i=6$ \\\hline\hline
  $n=1$ & $1$ &  &  &  &  \\\hline
  $n=2$ & $1!$ & $2!$ &  &  &  \\\hline
  $n=3$ & $2!$ & $6$ & $3!$ &  &  \\\hline
  $n=4$ & $3!$ & $22$ & $36$ & $4!$ & \\\hline
  $n=5$ & $4!$ & $100$ & $210$ & $240$ & $5!$ \\\hline
  $n=6$ & $5!$ & $548$ & $1350$ & $2040$ & $1800$ \\\hline
  $n=7$ & $6!$ & $3528$ & $9744$ & $17640$ & $21000$ \\\hline
  $n=8$ & $7!$ & $26136$ & $78792$ & $162456$ & $235200$ \\\hline
  $n=9$ & $8!$ & $219168$ & $708744$ & $1614816$ & $2693880$ \\\hline
  $n=10$ & $9!$ & $2053152$ & $7036200$ & $17368320$ & $32319000$ \\\hline
  $n=11$ & $10!$ & $21257280$ & $76521456$ & $201828000$ & $410031600$ \\
  \hline
  \hline
 $a_{n,i}$  & $i=7$ & $i=8$ & $i=9$ & $i=10$ & $i=11$ \\\hline\hline
  $n=6$ & $6!$ &  &  & &  \\\hline
  $n=7$ & $15120$ & $7!$ &  &  & \\\hline
  $n=8$ & $231840$ & $141120$ & $8!$ &  & \\\hline
  $n=9$ & $3265920$ & $2751840$ & $1451520$ & $9!$ &\\\hline
  $n=10$ & $45556560$ & $47628000$ & $35078400$ & $16329600$ & $10!$\\\hline
  $n=11$ & $649479600$ & $795175920$ & $731808000$ & $479001600$ & $199584000$\\
  \hline
\end{tabular}
\end{table}
Basing on the data listed in Table~\ref{coefficients}, we conjecture that the sequence $a_{n,i}$ for $n\in\mathbb{N}$ and $2\le i\le n+1$ is increasing with respect to $n$ while it is unimodal with respect to $i$.
\end{rem}

\begin{rem}
The elementary method and idea in the proof of Theorem~\ref{reciprocal=log=der=thm} has been employed in~\cite{derivative-tan-cot.tex} to establish an explicit formula for computing the $n$-th derivatives of the tangent and cotangent functions. This explicit formula for the $n$-th derivative of the cotangent function has been applied in~\cite{polygamma-sigularity.tex} to build the limit formulas for ratios of two polygamma functions at their singularities.
\end{rem}

\section{Explicit formula for Bernoulli numbers of the second kind}

In this section, basing on Theorem~\ref{reciprocal=log=der=thm}, we establish a new and explicit formula for calculating Bernoulli numbers $b_i$ of the second kind for $i\in\mathbb{N}$.

\begin{thm}\label{Bernulli-2rd=thm}
For $n\ge2$, Bernoulli numbers $b_n$ of the second kind can be computed by
\begin{equation}\label{Bernulli-2rd=formula}
b_n=(-1)^n\frac1{n!}\Biggl(\frac1{n+1}+\sum_{k=2}^{n} \frac{a_{n,k}-na_{n-1,k}}{k!}\Biggr),
\end{equation}
where $a_{n,k}$ are defined by~\eqref{a=n=0} and~\eqref{a=n=i=eq}.
\end{thm}

\begin{proof}
Differentiating the left-hand side of~\eqref{bernoulli-second-dfn} and making use of Theorem~\ref{reciprocal=log=der=thm} give
\begin{align*}
\biggl[\frac{x}{\ln(1+x)}\biggr]^{(i)}&=x\biggl[\frac1{\ln(1+x)}\biggr]^{(i)} +i\biggl[\frac1{\ln(1+x)}\biggr]^{(i-1)} \\
&=\frac{(-1)^ix}{(1+x)^i}\sum_{k=2}^{i+1}\frac{a_{i,k}}{[\ln(1+x)]^k} +\frac{(-1)^{i-1}i}{(1+x)^{i-1}}\sum_{k=2}^{i}\frac{a_{i-1,k}}{[\ln(1+x)]^k}\\
&=\frac{(-1)^i}{(1+x)^i}\Biggl\{x\sum_{k=2}^{i+1}\frac{a_{i,k}}{[\ln(1+x)]^k} -i(1+x)\sum_{k=2}^{i}\frac{a_{i-1,k}}{[\ln(1+x)]^{k}}\Biggr\}\\
&=\frac{(-1)^i}{(1+x)^i}\frac1{[\ln(1+x)]^{i+1}}\Biggl\{x\sum_{k=2}^{i+1}a_{i,k}[\ln(1+x)]^{i-k+1} \\
&\quad-i(1+x)\sum_{k=2}^{i}a_{i-1,k}[\ln(1+x)]^{i-k+1}\Biggr\}.
\end{align*}
Applying L'H\^ospital rule consecutively and by induction, we have
\begin{align*}
&\quad\lim_{x\to0}\frac{x\sum_{k=2}^{i+1}a_{i,k}[\ln(1+x)]^{i-k+1} -i(1+x)\sum_{k=2}^{i}a_{i-1,k}[\ln(1+x)]^{i-k+1}}{[\ln(1+x)]^{i+1}}\\
&=\lim_{u\to0}\frac{(e^u-1)\sum_{k=2}^{i+1}a_{i,k}u^{i-k+1} -ie^u\sum_{k=2}^{i}a_{i-1,k}u^{i-k+1}}{u^{i+1}}\\
&=\lim_{u\to0}\frac{a_{i,i+1}(e^u-1)-\sum_{k=2}^{i}a_{i,k}u^{i-k+1}+ \sum_{k=2}^{i}(a_{i,k}-ia_{i-1,k})\bigl(e^uu^{i-k+1}\bigr)}{u^{i+1}}\\
&=\frac1{(i+1)!}\lim_{u\to0}\Biggl[a_{i,i+1}(e^u-1)^{(i+1)} -\sum_{k=2}^{i}a_{i,k}\bigl(u^{i-k+1}\bigr)^{(i+1)}\\
&\quad+\sum_{k=2}^{i}(a_{i,k}-ia_{i-1,k})\bigl(e^uu^{i-k+1}\bigr)^{(i+1)}\Biggr]\\
&=\frac1{(i+1)!}\lim_{u\to0}\Biggl[a_{i,i+1}e^u +\sum_{k=2}^{i}(a_{i,k}-ia_{i-1,k})\bigl(e^uu^{i-k+1}\bigr)^{(i+1)}\Biggr]\\
&=\frac1{(i+1)!}\Biggl[a_{i,i+1} +\lim_{u\to0}\sum_{k=2}^{i}(a_{i,k}-ia_{i-1,k}) \sum_{m=0}^{i+1}\binom{i+1}{m}e^u\bigl(u^{i-k+1}\bigr)^{(m)}\Biggr]\\
&=\frac1{(i+1)!}\Biggl[i! +\sum_{k=2}^{i}(a_{i,k}-ia_{i-1,k}) \binom{i+1}{i-k+1}(i-k+1)!\Biggr]\\
&=\frac1{(i+1)!}\Biggl[i! +\sum_{k=2}^{i}(a_{i,k}-ia_{i-1,k}) \frac{(i+1)!}{k!}\Biggr]\\
&=\frac1{i+1}+\sum_{k=2}^{i} \frac{a_{i,k}-ia_{i-1,k}}{k!}.
\end{align*}
This means that
\begin{equation}\label{frac-limit=bernoulli-2rd}
\lim_{t\to0}\biggl[\frac{x}{\ln(1+x)}\biggr]^{(i)}=(-1)^i \Biggl(\frac1{i+1}+\sum_{k=2}^{i} \frac{a_{i,k}-ia_{i-1,k}}{k!}\Biggr).
\end{equation}
Differentiating the right-hand side of~\eqref{bernoulli-second-dfn} and taking limit generate
\begin{equation}\label{series-limit=bernoulli-2rd}
\lim_{x\to0}\Biggl[\Biggl(\sum_{n=0}^\infty b_nx^n\Biggr)^{(i)}\Biggr]
=\lim_{x\to0}\sum_{n=i}^\infty b_n\frac{n!}{(n-i)!}x^{n-i}=i!b_i.
\end{equation}
Equating~\eqref{frac-limit=bernoulli-2rd} and~\eqref{series-limit=bernoulli-2rd} leads to~\eqref{Bernulli-2rd=formula}. The proof of Theorem~\ref{Bernulli-2rd=thm} is complete.
\end{proof}

\begin{cor}\label{x-ln(1+x)-deriv-cor}
For $i\in\mathbb{N}$, we have
\begin{equation}\label{x-ln(1+x)-deriv-eq}
\biggl[\frac{x}{\ln(1+x)}\biggr]^{(i)} =\frac{(-1)^i}{(1+x)^i}\sum_{k=2}^{i+1}\frac{xa_{i,k}-i(1+x)a_{i-1,k}}{[\ln(1+x)]^k}
\end{equation}
and
\begin{equation}\label{x-ln(1+x)-deriv-eq-Stirling}
\biggl[\frac{x}{\ln(1+x)}\biggr]^{(i)} =\frac{(-1)^i}{(1+x)^i}\sum_{k=1}^{i}\frac{(-1)^{i+k}k![xs(i,k) +i(1+x)s(i-1,k)]}{[\ln(1+x)]^{k+1}},
\end{equation}
where $a_{i-1,i+1}=0$ and $s(i-1,i)=0$.
\end{cor}

\begin{proof}
The formula~\eqref{x-ln(1+x)-deriv-eq} can be deduced from the proof of Theorem~\ref{Bernulli-2rd=thm}.
Substituting~\eqref{a(n-i)=s(n=n-1)-eq} into~\eqref{x-ln(1+x)-deriv-eq} and simplifying result in~\eqref{x-ln(1+x)-deriv-eq-Stirling}. The proof is complete.
\end{proof}

\begin{rem}
The formula~\eqref{x-ln(1+x)-deriv-eq-Stirling} is a recovery and reformulation of~\cite[(10), Lemma~2]{Liu-Qi-Ding-2010-JIS}.
\end{rem}

\section{Integral representations of Stirling numbers of the first kind}

In this section, we will find several identities and integral representations relating to Stirling numbers of the first kind $s(n,k)$.

\begin{thm}\label{stirling-sum-gamma-thm}
For $1\le k\le n+1$, we have
\begin{equation}\label{stirling-sum-gamma}
\sum_{i=k-1}^n(-1)^{n+i}\frac{i!(i+1)!s(n,i)}{(i-k+1)!} =\int_0^\infty\frac{\Gamma(u+n)}{\Gamma(u)} \Biggl[\sum_{\ell=0}^{k-1}(-1)^\ell c_{k,\ell}u^{k-\ell}\Biggr]e^{-u}\td u,
\end{equation}
where $\Gamma(u)$ is the classical Euler gamma function which may be defined by
\begin{equation}\label{gamma-dfn}
\Gamma(z)=\int^\infty_0t^{z-1} e^{-t}\td t
\end{equation}
for $\Re z>0$ and
\begin{equation}\label{a-i-k-dfn}
c_{k,\ell}=\binom{k}{\ell}\binom{k-1}{\ell}{\ell!}
\end{equation}
for all $0\le\ell\le k-1$.
\end{thm}

\begin{proof}
In~\cite{Zhang-Li-Qi-Log.tex}, it was obtained that
\begin{equation}\label{frac1-ln-x+1-int}
\frac1{\ln(1+x)}=\int_0^\infty\frac1{(1+x)^u}\td u,\quad x>0.
\end{equation}
Utilizing this integral representation in~\eqref{Liu-Qi-Ding-2010-JIS-log-der} gives
\begin{equation*}
\int_0^\infty\frac{(-1)^m\Gamma(u+m)}{\Gamma(u)}\frac1{(1+t)^{u+m}}\td u =\frac1{(1+t)^m}\sum_{i=0}^m(-1)^ii!\frac{s(m,i)}{[\ln(1+t)]^{i+1}}.
\end{equation*}
Simplifying this yields
\begin{equation}\label{gamma-stirling}
\int_0^\infty\frac{\Gamma(u+m)}{\Gamma(u)}\frac1{(1+t)^{u}}\td u =\sum_{i=0}^m(-1)^{m+i}i!\frac{s(m,i)}{[\ln(1+t)]^{i+1}}.
\end{equation}
Substituting $t$ for $\frac1{\ln(1+t)}$ in~\eqref{gamma-stirling} brings out
\begin{equation}\label{gamma-stir-t}
\int_0^\infty\frac{\Gamma(u+m)}{\Gamma(u)}e^{-u/t}\td u =\sum_{i=0}^m(-1)^{m+i}i!s(m,i)t^{i+1}.
\end{equation}
Differentiating $1\le k\le m+1$ times with respect to $t$ on both sides of~\eqref{gamma-stir-t} generates
\begin{equation*}
\int_0^\infty\frac{\Gamma(u+m)}{\Gamma(u)}\bigl(e^{-u/t}\bigr)^{(k)}\td u =\sum_{i=k-1}^m(-1)^{m+i}i!s(m,i)\frac{(i+1)!}{(i-k+1)!}t^{i-k+1}.
\end{equation*}
Further letting $t\to1$ in the above equality produces
\begin{equation}\label{lim-to1-eq}
\sum_{i=k-1}^m(-1)^{m+i}i!s(m,i)\frac{(i+1)!}{(i-k+1)!} =\int_0^\infty\frac{\Gamma(u+m)}{\Gamma(u)} \lim_{t\to1}\Bigl[\bigl(e^{-u/t}\bigr)^{(k)}\Bigr]\td u.
\end{equation}
\par
In~\cite{exp-psi-cm-revised.tex} and~\cite[Theorem~2.2]{exp-reciprocal-cm-IJOPCM.tex}, it was obtained that
\begin{equation}\label{g(t)-derivative}
\bigl(e^{-1/t}\bigr)^{(i)}=\frac1{e^{1/t}t^{2i}} \sum_{k=0}^{i-1}(-1)^kc_{i,k}{t^{k}}
\end{equation}
for $i\in\mathbb{N}$ and $t\ne0$, where $c_{i,k}$ is defined by~\eqref{a-i-k-dfn}.
Combining this with
\begin{equation*}
\frac{\td{}^if(ut)}{\td t^i}=u^if^{(i)}(ut)
\end{equation*}
turns out
\begin{align*}
\bigl(e^{-u/t}\bigr)^{(k)}&=\frac{u^{k}}{e^{u/t}t^{2k}} \sum_{\ell=0}^{k-1}(-1)^\ell\frac{c_{k,\ell}}{u^\ell}{t^{\ell}}
\end{align*}
which tends to
\begin{equation*}
\frac{u^{k}}{e^{u}} \sum_{\ell=0}^{k-1}(-1)^\ell\frac{c_{k,\ell}}{u^\ell} =e^{-u} \sum_{\ell=0}^{k-1}(-1)^\ell c_{k,\ell}u^{k-\ell}
\end{equation*}
as $t\to1$. Substituting this into~\eqref{lim-to1-eq} builds~\eqref{stirling-sum-gamma}. Theorem~\ref{stirling-sum-gamma-thm} is proved.
\end{proof}

\begin{thm}\label{2rd-Bernoulli-final-thm}
For $1\le k\le m+1$, we have
\begin{multline}\label{2rd-Bernoulli-final-eq}
\sum_{i=k-1}^m\frac{(-1)^{m+i}i!(i+1)!s(m,i)}{(i-k+1)!} =m!\bigg\{\lim_{t\to1}\frac{\td{}^k}{\td t^k}\biggl[\frac{e^{m/t}}{(e^{1/t}-1)^{m+1}}\biggr] \\
+\int_1^\infty\frac1{[\ln(u-1)]^2+\pi^2} \lim_{t\to1}\frac{\td{}^k}{\td t^k}\biggl[\frac{{e^{m/t}}} {(e^{1/t}-1+u)^{m+1}}\biggr]\td u\biggr\}.
\end{multline}
\end{thm}

\begin{proof}
In~\cite{Zhang-Li-Qi-Log.tex}, it was recited that
\begin{equation}
\frac1{\ln(1+z)}=\frac1{z}+\int_1^\infty\frac1{[\ln(t-1)]^2+\pi^2} \frac{\td t}{z+t},\quad z\in\mathcal{A}.
\end{equation}
Here we remark that this formula corrects an error appeared in the proof of~\cite[Theorem~1.3, p.~2130]{Berg-Pedersen-ball-PAMS}. Therefore, by~\eqref{Liu-Qi-Ding-2010-JIS-log-der}, it is easy to see that
\begin{multline*}
(-1)^mm!\biggl[\frac1{t^{m+1}} +\int_1^\infty\frac1{[\ln(u-1)]^2+\pi^2} \frac{\td u}{(t+u)^{m+1}}\biggr]\\ =\frac1{(1+t)^m}\sum_{i=0}^m(-1)^ii!\frac{s(m,i)}{[\ln(1+t)]^{i+1}}.
\end{multline*}
Further replacing $\frac1{\ln(1+t)}$ by $t$ and rearranging reduce to
\begin{multline*}
\sum_{i=0}^m(-1)^{m+i}i!{s(m,i)}{t^{i+1}} =m!\biggl[\frac{e^{m/t}}{(e^{1/t}-1)^{m+1}} \\
+\int_1^\infty\frac1{[\ln(u-1)]^2+\pi^2} \frac{{e^{m/t}}\td u} {(e^{1/t}-1+u)^{m+1}}\biggr].
\end{multline*}
Differentiating $1\le k\le m+1$ times with respect to $t$ on both sides of the above equation creates
\begin{multline*}
\sum_{i=k-1}^m(-1)^{m+i}i!s(m,i)\frac{(i+1)!}{(i-k+1)!}t^{i-k+1}
=m!\bigg\{\frac{\td{}^k}{\td t^k}\biggl[\frac{e^{m/t}}{(e^{1/t}-1)^{m+1}}\biggr] \\
+\int_1^\infty\frac1{[\ln(u-1)]^2+\pi^2} \frac{\td{}^k}{\td t^k}\biggl[\frac{{e^{m/t}}} {(e^{1/t}-1+u)^{m+1}}\biggr]\td u\biggr\}.
\end{multline*}
Further letting $t\to1$ leads to Theorem~\ref{2rd-Bernoulli-final-thm}.
\end{proof}

\end{document}